\documentclass{amsart}
\usepackage{amsmath,amssymb,graphicx,amsthm,bm,hyperref,mathrsfs,tikz,cite}
\usepackage{graphicx}
\usepackage{hyperref}
\hypersetup{colorlinks=true,citecolor=blue,linkcolor=cyan}
\newtheorem{thm}{Theorem}[section]
\newtheorem{cor}[thm]{Corollary}
\newtheorem{lem}[thm]{Lemma}
\newtheorem{prop}[thm]{Proposition}
\newtheorem{conj}[thm]{Conjecture}
\theoremstyle{definition}

\theoremstyle{remark}
\newtheorem{rem}[thm]{Remark}
\numberwithin{equation}{section}

\newcommand{\Ff}{\mathbb{F}}

\newcommand{\Z}{\mathbb{Z}}
\newcommand{\F}{\mathcal{F}}

\newcommand{\Pp}{\mathbb{P}}
\newcommand{\Res}{\textnormal{Res}}

\newcommand{\Hh}{\mathcal{H}}
\newcommand{\G}{\mathcal{G}}
\newcommand{\Q}{\mathcal{Q}}

\newcommand{\supp}{\textnormal{supp}}
\newcommand{\lcm}{\textnormal{lcm}}
\newcommand{\Tr}{\textnormal{Tr}}
\newcommand{\dev}{\textnormal{dev}}
\newcommand{\biq}{\textnormal{biq}}

\begin{document}

\title[Expected Value of High Powers of Trace of Frobenius of Biquadratic Curves Over a Finite Field]{Expected Value of High Powers of Trace of Frobenius of Biquadratic Curves Over a Finite Field}%
\author{Patrick Meisner}%
%\address{}%
%\email{p_meisner@hotmail.com}%

%\thanks{}%
%\subjclass{}%
%\keywords{}%

%\date{}%
%\dedicatory{}%
%\commby{}%
% ----------------------------------------------------------------
\begin{abstract}

Denote $\Theta_C$ as the Frobenius class of a curve $C$ over the finite field $\Ff_q$. In this paper we determine the expected value of $\Tr(\Theta_C^n)$ where $C$ runs over all biquadratic curves when $q$ is fixed and $g$ tends to infinity. This extends work done by Rudnick \cite{R} and Chinis \cite{C} who separately looked at hyperelliptic curves and Bucur, Costa, David, Guerreiro and Lowry-Duda \cite{BCDG+} who looked at $\ell$-cyclic curves, for $\ell$ a prime, as well as cubic non-Galois curves.

\end{abstract}
\maketitle

\section{Introduction}

Let $C$ be a curve defined over $\Ff_q$, the finite field of $q$ elements, of genus $g$. Denote $N_n(C)$ as the number of $\Ff_{q^n}$-points on $C$. Then, for $|u|<q^{-1}$, the zeta function of $C$ is defined as
\begin{align}\label{zetafunc}
Z_C(u) = \exp\left(\sum_{n=1}^{\infty} N_n(C) \frac{u^n}{n}\right).
\end{align}
The Riemann Hypothesis, proved by Weil \cite{W}, states that $Z_C(u)$ is a rational function in $u$ of the form
\begin{align}\label{ratzetafunc}
Z_C(u) = \frac{P_C(u)}{(1-u)(1-qu)}
\end{align}
where $P_C(u)$ is a polynomial of degree $2g$ with integer coefficients satisfying the functional equation
$$P_C(u) = (qu^2)^gP_C\left(\frac{1}{qu}\right)$$
and whose roots lie on the circle $|u|=q^{-1/2}$.

These conditions imply that we can find a unique conjugacy class of a unitary symplectic $2g\times2g$ matrix, $\Theta_C$, such that $P_C(u) = \det(I - uq^{-1/2}\Theta_C)$. Hence the eigenvalues of $\Theta_C$ correspond to the zeros of $P_C(u)$. We call the conjugacy class of $\Theta_C$ the Frobenius class of $C$. Further, using this and equating \eqref{zetafunc} and \eqref{ratzetafunc} we can write
\begin{align}\label{numpts1}
N_n(C) = q^n+1-q^{n/2}\Tr(\Theta^n_C).
\end{align}
Therefore, studying statistics of $N_n(C)$ is equivalent to studying statistics of $\Tr(\Theta^n_C)$.

Katz and Sarnak \cite{KS} showed that for many families of curves of fixed genus $g$ the Frobenius class of $C$ becomes equidistributed in $USp(2g)$ as $q$ tends to infinity. There has been much research in recent years in examining the statistics if we fix $q$ and let the genus tend to infinity. This was first discussed by Kurlberg and Rudnick in \cite{KR} in which they considered the distribution of $\Tr(\Theta_C)$ over the family of hyperelliptic curves. This was extended by various authors to several different families \cite{BDFK+,BDFL1,BDFL2,LMM,M1,M2,M3}. For all cases, the statistics can be given as a sum of $q+1$ random vairables.

Later, Rudnick \cite{R} examined the expected value of $\Tr(\Theta_C^n)$ for the hyperelliptic ensemble as the genus tends to infinity and $q$ is odd. That is, let $\F_{2g+1}$ be the family of hyperelliptic curves given by the affine model
$$Y^2=F(X)$$
where $F$ is a square-free polynomial of degree $2g+1$. Then we have the following theorem.

\begin{thm}\label{rudnick}{\cite{R}}
Let $q$ be odd. If $3\log_q(g) < n < 4g-5\log_q(g)$ and $n\not=2g$, then as $g\to\infty$
$$\langle\Tr(\Theta_C^n)\rangle_{\F_{2g+1}} = \int_{USp(2g)} \Tr(U^n) dU + o\left(\frac{1}{g}\right).$$
\end{thm}

It is known exactly what these integrals evaluate to:
$$\int_{USp(2g)} \Tr(U^n) du = \begin{cases} -\eta_n & 1\leq n \leq 2g \\ 0 & n>2g  \end{cases}$$
where
$$\eta_n = \begin{cases} 1 & n \mbox{ is even} \\ 0 & n \mbox{ is odd} \end{cases}.$$
Therefore, since the theorem holds for $n<4g-5\log_q(g)$ we can see the dichotomy coming from $USp(2g)$ occuring in the family of curves.

Rudnick uses this to compute the one-level density. Let $f$ be an even Schwartz function and for any $N>1$ define
$$F(\theta) = \sum_{k\in\Z} f\left(N\left(\frac{\theta}{2\pi}-k\right)\right),$$
which has a period of $2\pi$. For a unitary $N\times N$ matrix $U$ with eigenvalues $e^{i\theta_j}$, $j=1,\dots,N$ denote
$$Z_f(U) = \sum_{j=1}^N F(\theta_j).$$

\begin{cor}{\cite{R}}
If $f$ is an even Schwartz function with Fourier transform $\hat{f}$ supported in $(-2,2)$, then as $g\to\infty$
$$\langle Z_f(\Theta_C)\rangle_{\F_{2g+1}} = \int_{USp(2g)} Z_f(U) dU + \frac{\dev(f)}{g} + o\left(\frac{1}{g}\right)$$
where
$$\dev(f) = \hat{f}(0)\sum_P \frac{\deg(P)}{|P|^2-1} - \frac{\hat{f}(1)}{q-1}$$
the sum being over all irreducible, monic polynomials in $\Ff_q[X]$.
\end{cor}

Chinis \cite{C} extends this to the whole moduli space. That is, let $\Hh_{g,2}$ denote the moduli space of hyperelliptic curves of genus $g$ (i.e curves with affine model $Y^2=F(X)$ where $F$ is a polynomial of degree $2g+1$ or $2g+2$).

\begin{thm} \label{chinis} {\cite{C}}
Let $q$ be odd. If $n$ is odd, then
$$\langle\Tr(\Theta_C^n)\rangle_{\Hh_{g,2}} = \int_{USp(2g)} \Tr(U^n) du.$$
For $n$ even and $3\log_q(g)<n<4g-5\log_q(g)$ then as $g\to\infty$
$$\langle\Tr(\Theta_C^n)\rangle_{\Hh_{g,2}} = \int_{USp(2g)} \Tr(U^n) du + o\left(1\right).$$
\end{thm}

\begin{rem}
Both Rudnick and Chinis gave more explicit results for the expectation over their families. They differ slightly, however they yield the same result in the $g$ limit.
\end{rem}

Bucur, Costa, David, Guerreiro and Lowry-Duda \cite{BCDG+} used a method different than Rudnick and Chinis to study the moduli space of cyclic $\ell$ covers of genus $g$, denoted $\Hh_{g,\ell}$, when $\ell$ is prime and $q\equiv 1 \mod{\ell}$. That is, they considered curves with affine model
$$Y^\ell = F(X)$$
of genus $g$.  In the case where $\ell=2$, they obtain the same results as Rudnick and Chinis only in a smaller range of $n$: $4\log_q(g)<n<2g(1-\epsilon)$, for some $\epsilon>0$. However, for $\ell$ an odd prime, they get a new result.

\begin{thm}\label{BCDG+}{\cite{BCDG+}}
Let $\ell$ be an odd prime and $q\equiv 1 \mod{\ell}$. For any $\epsilon>0$ and $n$ such that $6\log_q(g) < n < (1-\epsilon)\left(\frac{2g}{\ell-1}+2\right)$, as $g\to\infty$ we have
$$\langle\Tr(\Theta_C^n)\rangle_{\Hh_{g,\ell}} = \int_{U(2g)}\Tr(U^n)dU + O\left(\frac{1}{g}\right).$$
Let $f$ be an even Schwartz function with fourier coefficient $\hat{f}$ whose support is contained in $\left(\frac{-1}{\ell-1},\frac{1}{\ell-1}\right)$, then
$$\langle Z_f(\Theta_C)\rangle_{\Hh_{g,\ell}} = \int_{U(2g)} Z_f(U)dU - \hat{f}(0)\frac{\ell-1}{g} \sum_{v} \frac{\deg(v)}{(1+(\ell-1)q^{-\deg(v)})(q^{\ell\deg(v)/2}-1)} + O\left(\frac{1}{g^{2-\epsilon}}\right),$$
where the sum over all places $v$ of $\Ff_q(X)$.
\end{thm}

Again, it is known that
$$\int_{U(2g)} \Tr(U^n)dU = 0$$
for all $n$.

A similar theorem to Theorem \ref{BCDG+} is proved for the space of cubic, non-Galois extensions in \cite{BCDG+}. The symmetry type they determine is $USp(2g)$. However, since this paper will focus on Galois extensions we omit the statement of this result for brevity.

The focus of this paper will be on the moduli space of biquadratic curves of genus $g$, denoted $\Hh_{g,\biq}$, when $q$ is odd. That is, curves with affine models of the form
$$Y_1^2 = F_1(X) \quad \quad \quad Y_2^2 = F_2(X)$$
where $F_1$ and $F_2$ are square-free polynomials such that their degrees satisfy the genus formula (see \eqref{genform} below).

Finally, before stating the main theorem define
$$USp(2g)^3 = \left\{\begin{pmatrix} U_1 & 0 & 0 \\ 0 & U_2 & 0 \\ 0 & 0 & U_3 \end{pmatrix} : U_1,U_2,U_3 \in USp(2g)\right\} \subset USp(6g).$$
Therefore,
$$\int_{USp(2g)^3} \Tr(U^n) dU = 3\int_{USp(2g)} \Tr(U^n) dU = \begin{cases} -3\eta_n & n\leq 2g \\ 0 & n>2g \end{cases}.$$

\begin{thm}\label{mainthm}
Let $q$ be odd. For $n$ odd,
$$\langle\Tr(\Theta_C^n)\rangle_{\Hh_{g,\biq}} = \int_{USp(2g)^3} \Tr(U^n) dU.$$
For $n$ even and $3\log_q(g)<n=o\left(g^{2/3}\right)$, then as $g\to\infty$
$$\langle\Tr(\Theta_C^n)\rangle_{\Hh_{g,\biq}} = \int_{USp(2g)^3} \Tr(U^n) dU + o(1).$$
\end{thm}

Since our range of $n$ is sublinear in $g$, we can not prove a one-level density result. However, this result gives an obvious conjecture.

\begin{conj}
If $f$ is an even Schwartz function with Fourier transform $\hat{f}$ supported in $(-\alpha,\alpha)$, then as $g\to\infty$
$$\langle Z_f(\Theta_C)\rangle_{\Hh_{g,\biq}} = \int_{USp(2g)^3} Z_f(U) dU + o(1).$$
\end{conj}

In Section \ref{density} we prove the conjecture under the assumption that we can increase the range for $n$ in Theorem \ref{mainthm} to $n<2\alpha g$.

\textbf{Acknowledgements} I would like to thank Ze\'ev Rudnick for many useful conversations throughout this project.

The research leading to these results has received funding from the European Research Council under the European Union's Seventh Framework Programme (FP7/2007-2013) / ERC grant agreement n$^{\text{o}}$ 320755.

\section{Background on Function Fields}

In this section we give background on relevant results in function fields. Unless otherwise stated, the reference for this section will be \cite{rose}.

\subsection{Zeta and $L$-Function}

If we let $C=\mathbb{A}^1_{\Ff_q}$, the affine line, then \eqref{zetafunc} tells us that for $|u|<q^{-1}$
\begin{align}\label{zetafunc2}
Z(u) := Z_{\mathbb{A}^1_{\Ff_q}}(u) = \exp\left(\sum_{n=1}^{\infty} \frac{(qu)^n}{n}\right) = \frac{1}{1-qu}.
\end{align}
It is easy to see that for $|u|<q^{-1}$, $Z(u)$ has the following equivalent forms
\begin{align}\label{zetafunc3}
Z(u) = \sum_{F} u^{\deg(F)} = \prod_P \left( 1-u^{\deg(P)} \right)^{-1}
\end{align}
where the sum is over all monic polynomials in $\Ff_q[X]$ and the product is over all monic prime polynomials in $\Ff_q[X]$. We will use the change of variable $u=q^{-s}$ and define
\begin{align}\label{zetafuncs}
\zeta_q(s) := Z(u).
\end{align}

Let $\chi_D$ be a non-trivial quadratic Dirichlet character modulo $D$. Define the $L$-function related to this character as
\begin{align}\label{Lfunc}
L(u,\chi_D) = \sum_{F} \chi_D(F)u^{\deg(F)} = \prod_P \left(1-\chi_D(P)u^{\deg(P)}\right)^{-1}
\end{align}
which converges for $|u|<q^{-1}$. However,
$$\sum_{\deg(F)=d} \chi_D(F) = 0$$
for $d\geq \deg(D)$. Hence,
\begin{align}\label{Lfunc2}
L(u,\chi_D) = \sum_{d=0}^{\deg(D)-1} \left(\sum_{\deg(F)=d}\chi_D(F)\right) u^d
\end{align}
is a polynomial of degree at most $\deg(D)-1$.

For every Dirichlet character $\chi_D$, define $\chi_D^-$ as
\begin{align}\label{chipm}
\chi_D^{-}(F) = (-1)^{\deg(F)}\chi_D(F).
\end{align}
Define $L(u,\chi_D^-)$ in the same way as above. Then we have that
\begin{align}\label{Lfuncpm}
L(u,\chi_D^-) = L(- u, \chi_D)
\end{align}
and all relevant results that apply to $\chi_D$ also apply to $\chi^-_D$.

For consistency, we will write $\chi_D = \chi_D^+$ and often write $\chi_D^{\pm}$ to mean either $\chi_D^+$ or $\chi_D^-$.

\subsection{Prime and Square-free Polynomials}

The number of monic polynomials of degree $d$ can easily shown to be $q^d$. Let
$$\pi_q(n) = |\{P\in\Ff_q[X] : \deg(P)=n, P \mbox{ is a monic prime}\}|.$$
Then the Prime Polynomial Theorem says that
\begin{align}\label{PPNT}
\pi_q(n) = \frac{q^n}{n}+O\left(\frac{q^{n/2}}{n}\right).
\end{align}

Define the mobius function:
$$\mu(F) = \begin{cases} (-1)^r & F=P_1\cdots P_r, P_i \mbox{ are distinct primes} \\ 0 & \mbox{otherwise}\end{cases}.$$
Then $\mu^2$ is the indicator function for a polynomial being square-free. A simple sieving argument shows that
\begin{align}\label{squarefree}
\sum_{\deg(F)=d} \mu^2(F) = \begin{cases}\frac{q^d}{\zeta_q(2)} & d\geq 2 \\ q^d & d=0,1\end{cases}.
\end{align}

\subsection{Bounds of Character Sums}

Let $\chi_D$ be a non-trivial quadratic Dirichlet character modulo $D$. Let
$$\lambda = \begin{cases} 1 & \deg(D) \mbox{ even} \\ 0 & \deg(D) \mbox{ odd} \end{cases}.$$
Define
\begin{align}\label{L*}
L^*(u,\chi_D) = (1-u)^{\lambda} L(u,\chi_D)
\end{align}
and
$$2\delta= \deg(D)-1-\lambda.$$
Then $L^*(u,\chi_D)$ satisfies the function equation
$$L^*(u,\chi_D) = (qu^2)^{\delta}L^*\left(\frac{1}{qu},\chi_D\right)$$
and all the zeros of $L^*(u,\chi_D)$ lie on the circle $|u|=q^{-1/2}$ (this is the Riemann Hypothesis for $L$-functions \cite{W}). Therefore, we can find a matrix $\Theta_D\in USp(2\delta)$ such that
\begin{align}\label{Lfuncdet}
L^*(u,\chi_D) = \det(I-u\sqrt{q}\Theta_D).
\end{align}

Substituting \eqref{Lfunc} into \eqref{L*}, equating it with \eqref{Lfuncdet} and taking logarithmic derivatives we obtain the formula
$$-\Tr(\Theta_{D}^n) = \frac{\lambda}{q^{n/2}} + \frac{1}{q^{n/2}}\sum_{\deg(F)=n} \Lambda(F) \chi_D(F)$$
where $\Lambda(F)$ is the von Mangoldt function.

Therefore, since $\Theta_D$ is unitary we get
$$\left|\sum_{\deg(P)=n} \chi_D(P)\right| \ll \frac{\deg(D)}{n}q^{n/2}$$
where the sum is over prime polynomials.

This can obviously be extended to $\chi_D^{\pm}$:
\begin{align}\label{Riemhypbound}
\left|\sum_{\deg(P)=n} \chi_D^{\pm}(P)\right| \ll \frac{\deg(D)}{n}q^{n/2}.
\end{align}

\subsection{Bounds of Double Character Sums}

Rudnick \cite{R} for $d$ odd and Chinis \cite{C} for $d$ even showed that as long as $n<2d-C\log_q(g)$ for some $C$, then
\begin{align}\label{doublecharsum}
\frac{n}{q^{d+n/2}}\sum_{\deg(D)=d}\sum_{\substack{\deg(P)=n \\ P \mbox{ prime}}} \mu^2(D) \chi_D(P) \ll 1
\end{align}
as $d\to\infty$.

Finding much better asymptotic results for the above is the main content of \cite{C} and \cite{R}. However, for our purposes what we have stated above suffices.

\section{Trace Formula}

As mentioned in the introduction, the curves in $\Hh_{g,\biq}$ have affine model of the form
\begin{align}\label{affmodel}
Y_1^2 = F_1(X) \quad\quad\quad Y_2^2=F_2(X)
\end{align}
where $F_1$ and $F_2$ are square-free polynomials. Let $f_3 = \gcd(F_1,F_2)$. Then we can write $F_1 = f_1f_3$ and $F_2 = f_2f_3$. Since $F_1$ and $F_2$ were square-free, we get that $f_1,f_2,f_3$ are square-free and pairwise coprime.

Equation \eqref{numpts1} shows that knowing $\Tr(\Theta_C^n)$ is equivalent to knowing $N_n(C)$, the number of $\Ff_{q^n}$-points on $C$. A formula for $N_1(C)$ is developed in \cite{LMM} and \cite{M2}. The same argument works for $N_n(C)$. That is,
$$N_n(C) = \sum_{x\in\Pp^1(\Ff_{q^n})} \left(1+\chi_2(f_1f_3(x)) + \chi_2(f_2f_3(x)) + \chi_2(f_1f_2(x))\right)$$
where $\chi_2:\Ff_{q^n} \to \mathbb{C}$ is a non-trivial quadratic character and, if we denote the point at infinity as $\infty$, we define
$$F(\infty) = \begin{cases} \mbox{leading coefficient of $F$} & \deg(F)\equiv 0 \mod{2} \\ 0 & \mbox{otherwise} \end{cases}.$$
Hence, we have
\begin{align}\label{traceform1}
\Tr(\Theta_C^n) = -q^{-n/2}\sum_{x\in\Pp^1(\Ff_{q^n})} \left(\chi_2(f_1f_3(x)) + \chi_2(f_2f_3(x)) + \chi_2(f_1f_2(x))\right).
\end{align}

It is shown in \cite{LMM} and \cite{M2} that if we define
\begin{align}\label{genfunL}
L(d_1,d_2,d_3) = d_1+d_2+d_3 + \begin{cases} 0 & d_1+d_3\equiv d_2+d_3\equiv 0\mod{2} \\ 1 & \mbox{otherwise}\end{cases}
\end{align}
then the genus of the curves in $\Hh_{g,\biq}$ satisfies the equation
\begin{align}\label{genform}
g+3 =L(\deg(f_1),\deg(f_2),\deg(f_3)).
\end{align}

We would like to say that all biquadratic curves come from the triples of polynomials such that $L(\deg(f_1),\deg(f_2),\deg(f_3))=g+3$. However, this is not the case. If $g$ is odd then we have $L(g+3,0,0)=L(g+2,0,0)=g+3$. So if we have a triple of polynomials such that $\{\deg(f_1),\deg(f_2),\deg(f_3)\}=\{g+3,0,0\}$ or $\{g+2,0,0\}$, then two of the three polynomials would be constants and this would correspond to a hyperelliptic curve and not a biquadratic curve. It is easy to verify that these are the only instances in which $g+3 =L(\deg(f_1),\deg(f_2),\deg(f_3))$ is satisfied but we do not obtain a biquadratic curve.

\begin{rem}
To save space, for a triple of polynomial $f_1,f_2,f_3$ we will write
\begin{align}\label{genfunL2}
L(f_1,f_2,f_3) := L(\deg(f_1),\deg(f_2),\deg(f_3)).
\end{align}
\end{rem}

With this in mind, define the set of triples of polynomials
\begin{align}\label{Fsethat}
\begin{split}
\widehat{\F}_g = \{(f_1,f_2,f_3)  : & f_1,f_2,f_3 \mbox{ are square-free and pairwise coprime, } f_3 \mbox{ is monic, }\\
&   L(f_1,f_2,f_3)=g+3 \mbox{ and } \{\deg(f_1),\deg(f_2),\deg(f_3)\} \not= \{g+3,0,0\},\{g+2,0,0\} \}.
\end{split}
\end{align}
Then $\widehat{\F_g}$ parameterizes $\Hh_{g,\biq}$.

\begin{rem}\label{traceformrem}
As noted above the condition $L(f_1,f_2,f_3)=g+3$ and  $\{\deg(f_1),\deg(f_2),\deg(f_3)\} = \{g+3,0,0\}$ or $\{g+2,0,0\}$ can only be simultaneously obtained when $g$ is odd. Hence, if $g$ is even, we may omit the condition  $\{\deg(f_1),\deg(f_2),\deg(f_3)\} \not= \{g+3,0,0\}$ or $\{g+2,0,0\}$.
\end{rem}

\begin{rem}
If $f_3$ is allowed to be non-monic in $\widehat{\F}_g$ then each tuple would correspond to $q-1$ different pairs of $F_1,F_2$ in \eqref{affmodel}, which we want to avoid.
\end{rem}

Therefore, we can write the expected trace as
\begin{align}
\langle\Tr(\Theta_C^n)\rangle_{\Hh_{g,\biq}} = \frac{-q^{-n/2}}{|\widehat{\F}_g|} \sum_{(f_1,f_2,f_3)\in\widehat{\F}_g} \sum_{x\in\Pp^1(\Ff_{q^n})} \left(\chi_2(f_1f_3(x)) + \chi_2(f_2f_3(x)) + \chi_2(f_1f_2(x))\right).
\end{align}

\section{Main Term}

In this section, we calculate the main term of $\Tr(\Theta_C^n)$ as well as rewrite the formula in terms of a sum over monic poynomials. With this in mind define the set
\begin{align}\label{Fset}
\F_g = \{(f_1,f_2,f_3)\in\widehat{\F}_g : f_1,f_2,f_3 \mbox{ are monic}\}.
\end{align}

\subsection{Odd n}\label{Oddn} If $n$ is odd then $\chi_2$ is a non-trivial quadratic character on $\Ff_q$. Therefore,
$$\sum_{\alpha\in\Ff^*_q} \chi_2(\alpha) =0.$$

For any $x\in\Ff_{q^n}$,
\begin{align*}
& \sum_{(f_1,f_2,f_3)\in\widehat{\F}_g} \left(\chi_2(f_1f_3(x)) + \chi_2(f_2f_3(x)) + \chi_2(f_1f_2(x))\right)\\
& =  \sum_{(f_1,f_2,f_3)\in\F_g}\sum_{\alpha,\beta\in\Ff^*_q}  \left(\chi_2(\alpha f_1f_3(x)) + \chi_2(\beta f_2f_3(x)) + \chi_2(\alpha\beta f_1f_2(x))\right) \\
& = \sum_{(f_1,f_2,f_3)\in\F_g}  \left(\chi_2(f_1f_3(x))\sum_{\alpha,\beta\in\Ff^*_q}\chi_2(\alpha) + \chi_2(f_2f_3(x))\sum_{\alpha,\beta\in\Ff^*_q}\chi_2(\beta) + \chi_2(f_1f_2(x))\sum_{\alpha,\beta\in\Ff^*_q}\chi_2(\alpha\beta)\right) \\
& = 0.
\end{align*}
Hence, in this case we get
$$\langle \Tr(\Theta_C^n) \rangle_{\Hh_{g,\biq}} = 0.$$

Now, we notice that for odd $n$
$$\int_{USp(2g)} \Tr(U^n) du = 0.$$
Therefore, this proves the first statement of Theorem \ref{mainthm}.

From now on we assume $n$ is even.

\subsection{Even n}

If $n$ is even and $\alpha\in\Ff^*_{q^{n/2}}$ then $[\Ff_{q^{n/2}}(\sqrt{\alpha}):\Ff_{q^{n/2}}] = 1$ or $2$. That is,
$$\Ff_{q^{n/2}}(\sqrt{\alpha}) = \Ff_{q^{n/2}}\subset \Ff_{q^n} \mbox{ or } \Ff_{q^{n/2}}(\sqrt{\alpha})=\Ff_{q^n}.$$
Hence, $\sqrt{\alpha} \in \Ff_{q^n}$ and so $\chi_2(\alpha)=1$. Therefore, for any $x\in\Pp^1(\Ff_{q^n})$
\begin{align*}
& \sum_{(f_1,f_2,f_3)\in\widehat{\F}_g} \left(\chi_2(f_1f_3(x)) + \chi_2(f_2f_3(x)) + \chi_2(f_1f_2(x))\right)\\
& = (q-1)^2\sum_{(f_1,f_2,f_3)\in\F_g}\left(\chi_2(f_1f_3(x)) + \chi_2(f_2f_3(x)) + \chi_2(f_1f_2(x))\right) \\
& = 3(q-1)^2\sum_{(f_1,f_2,f_3)\in\F_g}\chi_2(f_1f_2(x))
\end{align*}
with the last equality comes from the fact that the summand over $\F_g$ is symmetric in the arguments of $f_1,f_2,f_3$.

Moreover, if $F$ is any polynomial with coefficients in $\Ff_q$ and $x\in\Ff_{q^{n/2}}$ then $F(x)\in\Ff_{q^{n/2}}$ and $\chi_2(F(x))=1$ if $x$ is not a root of $F$ and $0$ otherwise. Therefore,
$$ \sum_{x\in\Ff_{q^{n/2}}}\sum_{(f_1,f_2,f_3)\in\F_g}\chi_2(f_1f_2(x))  = \sum_{x\in\Ff_{q^{n/2}}}\sum_{\substack{(f_1,f_2,f_3)\in\F_g \\ f_1f_2(x)\not=0}} 1 $$
$$= q^{n/2}|\F_g| - \sum_{x\in\Ff_{q^{n/2}}}\sum_{\substack{(f_1,f_2,f_3)\in\F_g \\ f_1f_2(x)=0}} 1 $$

Now, we notice that $|\widehat{\F}_g| = (q-1)^2|\F_g|$ and therefore, we get that
\begin{align}\label{traceform2}
\begin{split}
\langle\Tr(\Theta_C^n)\rangle_{\Hh_{g,\biq}} = & -3 + \frac{3q^{-n/2}}{|\F_g|}\sum_{x\in\Ff_{q^{n/2}}} \sum_{\substack{(f_1,f_2,f_3)\in\F_g \\ f_1f_2(x)=0}} 1 \\  & - \frac{3q^{-n/2}}{|\F_g|} \sum_{\substack{x\in\Ff_{q^n} \\ x\not\in\Ff_{q^{n/2}} }} \sum_{(f_1,f_2,f_3)\in\F_g} \chi_2(f_1f_2(x)).
\end{split}
\end{align}
So, we see we have a constant term of $-3$ appearing. This is the main term and the rest of this paper is devoted to determining when the other terms tend to $0$.

\section{Error Terms: Easy Cases}

In this section we show that as long as $n$ is sufficiently large then many terms in \eqref{traceform2} tend to zero as the genus tends to infinity.

\subsection{Roots}

First, let us deal with the term coming from the roots of $f_1f_2$:
$$\frac{3q^{-n/2}}{|\F_g|}\sum_{x\in\Ff_{q^{n/2}}} \sum_{\substack{(f_1,f_2,f_3)\in\F_g \\ f_1f_2(x)=0}} 1 = \frac{3q^{-n/2}}{|\F_g|} \sum_{(f_1,f_2,f_3)\in\F_g } \sum_{\substack{x\in\Ff_{q^{n/2}} \\ f_1f_2(x)=0}} 1$$

Now, given an $(f_1,f_2,f_3)\in\F_g$, how many roots can $f_1f_2$ have in $\Ff_{q^{n/2}}$? Since $\deg(f_1f_2)\leq g+3$, there are at most $g+3$ roots. Hence the above will be less than
$$\frac{3q^{-n/2}}{|\F_g|} \sum_{(f_1,f_2,f_3)\in\F_g }(g+3) = 3\frac{g+3}{q^{n/2}}.$$
So then as long as $n\geq (2+\epsilon)\log_q(g)$, we get
\begin{align}\label{roots}
\frac{3q^{-n/2}}{|\F_g|}\sum_{x\in\Ff_{q^{n/2}}} \sum_{\substack{(f_1,f_2,f_3)\in\F_g \\ f_1f_2(x)=0}} 1 = o\left(1\right)
\end{align}
as $g\to\infty$.

\subsection{Non-generating Values}

Now let us consider the contributions in the second error term of \eqref{traceform2} that come from values that do not generate $\Ff_{q^n}$ over $\Ff_q$. That is the sum
\begin{align}\label{nongenval}
\frac{3q^{-n/2}}{|\F_g|} \sum_{\substack{x\in\Ff_{q^n} \\ x\not\in\Ff_{q^{n/2}} \\ \deg(P_x)\not=n }} \sum_{(f_1,f_2,f_3)\in\F_g} \chi_2(f_1f_2(x))
\end{align}
where we denote $P_x$ as the minimum polynomial of $x$ over $\Ff_q$. Now, we will use the very trivial bound $\chi_2(f_1f_2(x)) \leq 1$ to get that \eqref{nongenval} is less than
\begin{align}\label{nongenval2}
3q^{-n/2} \sum_{\substack{x\in\Ff_{q^n} \\ x\not\in\Ff_{q^{n/2}} \\ \deg(P_x)\not=n }} 1.
\end{align}

All the $x$'s in the sum have the property that $\deg(P_x)|n$ (since $x\in\Ff_{q^n}$), $\deg(P_x)\nmid \frac{n}{2}$ (since $x\not\in\Ff_{q^{n/2}}$) and $\deg(P_x)\not=n$. Hence all the $x$'s are contained in $\Ff_{q^D}$ where
$$D = \lcm\left(d : d|n, d\nmid \frac{n}{2}, d\not=n\right)\leq n/3.$$

Therefore, there are $q^D$ terms appearing in the sum and we get \eqref{nongenval2} will be less than
$$3q^{-n/2+D} \leq 3q^{-n/6}.$$
As long as $n$ tends to infinity with $g$, we have
\begin{align}\label{nongenval3}
\frac{3q^{-n/2}}{|\F_g|} \sum_{\substack{x\in\Ff_{q^n} \\ x\not\in\Ff_{q^{n/2}} \\ \deg(P_x)\not=n }} \sum_{(f_1,f_2,f_3)\in\F_g} \chi_2(f_1f_2(x)) = o\left(1\right)
\end{align}
as $g\to\infty$.

\section{Error Term: Fixed Prime}

It remains to determine how quickly
\begin{align}\label{errterm1}
\frac{3q^{-n/2}}{|\F_g|} \sum_{\substack{x\in\Ff_{q^n} \\  \deg(P_x)=n }} \sum_{(f_1,f_2,f_3)\in\F_g} \chi_2(f_1f_2(x))
\end{align}
grows as $g$ tends to infinity.

\subsection{Reformulation}

We first notice that for all $x$ such that $\deg(P_x)=n$, there is an isomorphism
\begin{center}
\begin{tabular}{c c c}
$\Ff_q[X]/(P_x)$ & $\to$ & $\Ff_{q^n}$\\
$\overline{F}$ & $\mapsto$ & $F(x)$.
\end{tabular}
\end{center}
Thus, $F(x)$ is a square in $\Ff_{q^n}$ if and only if $F$ is a square modulo $P_x$. Hence,
$$\chi_2(F(x)) = \chi_{P_x}(F) := \left(\frac{F}{P_x}\right)$$
where $\left(\frac{\cdot}{\cdot}\right)$ is the Legendre symbol.

Finally, there is an obvious $n$-to-$1$ correspondence between $x\in \Ff_{q^n}$ such that $\deg(P_x)=n$ and prime polynomials in $\Ff_q[X]$ of degree $n$. Thus, we can rewrite $\eqref{errterm1}$ as
\begin{align}\label{errterm2}
\frac{3nq^{-n/2}}{|\F_g|} \sum_{\deg(P)=n} \sum_{(f_1,f_2,f_3)\in\F_g} \chi_P(f_1f_2).
\end{align}

\subsection{Fixed Prime}\label{fixedprime}

In this section, we will fix a prime $P$ and determine the sum
$$\sum_{(f_1,f_2,f_3)\in\F_g} \chi_P(f_1f_2)$$

Since the definition of $\F_g$ depends on $L(d_1,d_2,d_3)$ which itself depends on the congruence of $d_1,d_2,d_3$ modulo $2$, we will first determine the sum
$$N_{k_1,k_2}(d;P) := \sum_{\substack{\deg(f_1f_2f_3)=d \\ \deg(f_1f_3)\equiv k_1\mod{2} \\ \deg(f_2f_3) \equiv k_2 \mod{2}}} \mu^2(f_1f_2f_3)\chi_P(f_1f_2)$$
for a given $k_1,k_2\in\{0,1\}$.

For any prime polynomial $P$, recall the definition $\chi_P^{\pm}$ and $L(u,\chi_P^{\pm})$ from \eqref{Lfunc} and \eqref{chipm}. Then we have the following result.

\begin{lem} \label{fixedprimelem}
For a fixed prime polynomial $P$ and $k_1,k_2 \in \{0,1\}$, we have
$$N_{k_1,k_2}(d;P) = \frac{C_{k_1,k_2}(d;P)}{4} q^d + O\left((q^{\frac{1}{2}+\epsilon)d}\right)$$
where $C_{k_1,k_2}(d;P)$ is defined as
\begin{align}\label{C_k(P)}
\begin{split}
C_{k_1,k_2}(d;P) & = L(q^{-1},\chi_P^+)^2H_{P,+}(q^{-1}) + (-1)^{k_1+k_2}L(q^{-1},\chi_P^-)^2H_{P,-}(q^{-1}) \\
& + (-1)^d((-1)^{k_1}+(-1)^{k_2})L(q^{-1},\chi_P^+)L(q^{-1},\chi_P^-)H_{P,0}(q^{-1})
\end{split}
\end{align}
and $H_{P,+}(u),H_{P,-}(u)$ and $H_{P,0}(u)$ are functions that converge for $|u|<q^{-1/2}$ and are defined in \eqref{Hpmdef} and \eqref{H0def}.
\end{lem}

\begin{proof}

Consider the generating series
\begin{align*}
\G_{k_1,k_2}(u) & := \sum_{\substack{f_1,f_2,f_3 \\ \deg(f_1f_3)\equiv k_1\mod{2} \\ \deg(f_2f_3) \equiv k_2 \mod{2}}} \mu^2(f_1f_2f_3)\chi_P(f_1f_2)u^{\deg(f_1f_2f_3)} \\
& = \sum_{f_1,f_2,f_3}  \mu^2(f_1f_2f_3)\chi_P(f_1f_2)\frac{1}{4}\left(1 + (-1)^{k_1+\deg(f_1f_3)}\right)\left(1+(-1)^{k_2+\deg(f_2f_3)}\right)u^{\deg(f_1f_2f_3)} \\
%& = \frac{1}{4}\sum_{f_1,f_2,f_3} \mu^2(f_1f_2f_3)\chi_P(f_1f_2)u^{\deg(f_1f_2f_3)} + \frac{(-1)^{k_1}}{4}\sum_{f_1,f_2,f_3} \mu^2(f_1f_2f_3)\chi_P(f_1f_2)(-u)^{\deg(f_1f_3)}u^{\deg(f_2)} \\
%& + \frac{(-1)^{k_2}}{4}\sum_{f_1,f_2,f_3} \mu^2(f_1f_2f_3)\chi_P(f_1f_2)(-u)^{\deg(f_2f_3)}u^{\deg(f_1)} + \frac{(-1)^{k_1+k_2}}{4}\sum_{f_1,f_2,f_3} \mu^2(f_1f_2f_3)\chi_P(f_1f_2)(-u)^{\deg(f_1f_2)}u^{\deg(f_3)}\\
& = \frac{1}{4}\left(S_+(u) + \left((-1)^{k_1}+(-1)^{k_2}\right)S_0(u) + (-1)^{k_1+k_2}S_-(u)\right)
\end{align*}
where
\begin{align*}
S_+(u) &:= \sum_{f_1,f_2,f_3} \mu^2(f_1f_2f_3)\chi_P(f_1f_2)u^{\deg(f_1f_2f_3)} = \prod_Q\left(1 + u^{\deg(Q)}+2\chi^+_P(Q)u^{\deg(Q)}\right)\\
S_-(u) & := \sum_{f_1,f_2,f_3} \mu^2(f_1f_2f_3)\chi_P(f_1f_2)(-u)^{\deg(f_1f_2)}u^{\deg(f_3)} = \prod_Q\left(1 + u^{\deg(Q)}+2\chi^-_P(Q)u^{\deg(Q)}\right)\\
S_0(u) & := \sum_{f_1,f_2,f_3} \mu^2(f_1f_2f_3)\chi_P(f_1f_2)(-u)^{\deg(f_1f_3)}u^{\deg(f_2)}\\
& = \prod_Q\left(1 + ((-1)^{\deg(Q)} + \chi_P^+(Q)+\chi_P^-(Q))u^{\deg(Q)}\right)
\end{align*}
and each of the products run over all monic prime polynomials in $\Ff_q[X]$.

Now,
\begin{align}\label{Hpmdef}
\begin{split}
S_{\pm}(u) & = \prod_Q\left(1 + u^{\deg(Q)}+2\chi^{\pm}_P(Q)u^{\deg(Q)}\right)\\
& = \prod_Q \left(1-u^{\deg(Q)}\right)^{-1} \prod_{Q}\left(1-\chi^{\pm}_P(Q) u^{\deg(Q)}\right)^{-2} H_{P,\pm}(u)\\
& = Z(u)L(u,\chi^{\pm}_P)^2 H_{P,\pm}(u)
\end{split}
\end{align}
where $Z(u)$ is the zeta function and $H_{P,\pm}(u)$ converges for $|u|<q^{-1/2}$. Further,
\begin{align}\label{H0def}
\begin{split}
S_0(u) & = \prod_Q \left(1+((-1)^{\deg(Q)}+\chi_P^+(Q)+\chi_P^-(Q))u^{\deg(Q)}\right)\\
& = \prod_Q \left(1-(-u)^{\deg(Q)}\right)^{-1}\prod_{Q}\left(1-\chi_P^+(Q)u^{\deg(Q)}\right)^{-1}\prod_{Q}\left(1-\chi_P^-(Q)u^{\deg(Q)}\right)^{-1} H_{P,0}(u)\\
& = Z(-u) L(u,\chi_P^+) L(u,\chi_P^-) H_{P,0}(u)
\end{split}
\end{align}
where $H_{P,0}(u)$ converges for $|u|<q^{-1/2}$.
%\begin{align}\label{Hpm}
%\begin{split}
%H_{\pm}(u) := \prod_Q \left( 1 - \left[2\chi^{\pm}_P(Q) + 3\chi^2_P(Q)\right]u^{2\deg(Q)} + \left[2\chi^{\pm}_P(Q)-4\chi^2_P(Q)\right]u^{3\deg(Q)} - 2\chi^{\pm}_P(Q)u^{4\deg(Q)} \right).
%\end{split}
%\end{align}

Therefore, we see that $\G_{k_1,k_2}(u)$ can be analytically continued to the region $|u|<q^{-1/2}$ with simple poles at $u=q^{-1},-q^{-1}$.

If we let $C_{\epsilon} = \{u : |u|=q^{-(1/2+\epsilon)}\}$ for some $\epsilon>0$, then by Cauchy's residue theorem
$$ \Res_{u=0}\left(\frac{\G_{k_1,k_2}(u)}{u^{d+1}}\right) + \Res_{u=q^{-1}}\left(\frac{\G_{k_1,k_2}(u)}{u^{d+1}}\right)+ \Res_{u=-q^{-1}}\left(\frac{\G_{k_1,k_2}(u)}{u^{d+1}}\right)$$
$$ = \oint_{C_{\epsilon}} \frac{\G_{k_1,k_2}(u)}{u^{d+1}} du \ll q^{(\frac{1}{2}+\epsilon)D}$$
By definition of $\G_{k_1,k_2}(u)$ we have
$$ \Res_{u=0}\left(\frac{\G_{k_1,k_2}(u)}{u^{d+1}}\right) = N_{k_1,k_2}(d;P).$$
Finally,
\begin{align*}
\Res_{u=q^{-1}}\left(\frac{\G_{k_1,k_2}(u)}{u^{d+1}}\right) & = \frac{1}{4}\Res_{u=q^{-1}}\left(\frac{S_+(u)}{u^{d+1}} + \frac{(-1)^{k_1+k_2}S_-(u)}{u^{d+1}} \right)\\
& = \frac{-q^d}{4} \bigg(L(q^{-1},\chi_P^+)^2H_{P,+}(q^{-1}) + (-1)^{k_1+k_2}L(q^{-1},\chi_P^-)^2H_{P,-}(q^{-1})\bigg)
\end{align*}
and
\begin{align*}
\Res_{u=-q^{-1}}\left(\frac{\G_{k_1,k_2}(u)}{u^{d+1}}\right) & = \frac{(-1)^{k_1}+(-1)^{k_2}}{4}\Res_{u=-q^{-1}}\left(\frac{S_0(u)}{u^{d+1}}\right)\\
&  = \frac{- q^d(-1)^d((-1)^{k_1}+(-1)^{k_2})}{4} L(-q^{-1},\chi^+_P)L(-q^{-1},\chi^-_P)H_{P,0}(-q^{-1}).
\end{align*}
We use the observation that $L(u,\chi_P^-)=L(-u,\chi_P^+)$ and consqeuntly $H_{P,0}(-u)=H_{P,0}(u)$ to write it in the form in the statement.

\end{proof}

\begin{prop}\label{fixedprimeprop}

For a fixed prime polynomial $P$, we have
$$\sum_{(f_1,f_2,f_3)\in\F_g} \chi_P(f_1f_2) = \frac{C(g;P)}{4} q^{g+3} - (1-\eta_g)\left(2\sum_{\deg(f)=g+2,g+3} \mu^2(f)\chi_P(f) +\frac{q+1}{q}\frac{q^{g+3}}{\zeta_q(2)}\right) + O\left(q^{(\frac{1}{2}+\epsilon)g}\right)$$
where
\begin{align}\label{C(P)}
\begin{split}
C(g;P) & =  \frac{q+3}{q}L(q^{-1},\chi_P^+)^2H_{P,+}(q^{-1}) + \frac{q-1}{q}L(q^{-1},\chi_P^-)^2H_{P,-}(q^{-1}) \\
& - 2(-1)^g\frac{q+1}{q}L(q^{-1},\chi_P^+)L(q^{-1},\chi_P^-)H_{P,0}(q^{-1})
\end{split}
\end{align}
\end{prop}

\begin{proof}

First, note that
\begin{align*}
& \sum_{\substack{f_1,f_2,f_3 \\ L(f_1,f_2,f_3)=g+3}} \mu^2(f_1f_2f_3)\chi_P(f_1f_2)\\
& = N_{0,0}(g+3;P) + N_{0,1}(g+2;P)+N_{1,0}(g+2;P) + N_{1,1}(g+2;P) \\
& = \frac{C(g;P)}{4} q^{g+3} + O\left(q^{(\frac{1}{2}+\epsilon)g}\right)
\end{align*}

If $g$ is odd, then by Remark \ref{traceformrem} we need to exclude the cases where $\{\deg(f_1),\deg(f_2),\deg(f_3)\} = \{g+3,0,0\}$ or $\{g+2,0,0\}$.
\begin{align*}
\sum_{(f_1,f_2,f_3)\in\F_g} \chi_P(f_1f_2) & = \sum_{\substack{f_1,f_2,f_3 \\ L(f_1,f_2,f_3)=g+3}} \mu^2(f_1f_2f_3)\chi_P(f_1f_2) - \sum_{\deg(f)=g+2,g+3} \mu^2(f)(2\chi_P(f)+1)\\
& = \frac{C(g;P)}{4} q^{g+3} + O\left(q^{(\frac{1}{2}+\epsilon)g}\right) - 2\sum_{\deg(f)=g+2,g+3} \mu^2(f)\chi_P(f) - \frac{q+1}{q}\frac{q^{g+3}}{\zeta_q(2)}
\end{align*}
where the last term comes from \eqref{squarefree}.

If $g$ is even, then by Remark \ref{traceformrem} we do not need to exclude anything and we get
\begin{align*}
\sum_{(f_1,f_2,f_3)\in\F_g} \chi_P(f_1f_2) & = \sum_{\substack{f_1,f_2,f_3 \\ L(f_1,f_2,f_3)=g+3}} \mu^2(f_1f_2f_3)\chi_P(f_1f_2)\\
& = \frac{C(g;P)}{4} q^{g+3} +  O\left(q^{(\frac{1}{2}+\epsilon)g}\right)
\end{align*}

\end{proof}

Now we will take the result of Proposition \ref{fixedprimeprop} and sum over all primes. Corollary 1.7 of \cite{M3} shows that there exists a quadratic polynomials, $A$, such that
\begin{align}\label{setsize}
|\F_g| = A(g)q^{g+3}+O\left(q^{\left(\frac{1}{2}+\epsilon\right)g}\right).
\end{align}
Hence by \eqref{doublecharsum} we have, for $n<2g-C\log_q(g)$,
$$\frac{3nq^{-n/2}}{|\F_g|}\sum_{\deg(P)=n}\left(2\sum_{\deg(f)=g+2,g+3} \mu^2(f)\chi_P(f) +\frac{q+1}{q}\frac{q^{g+3}}{\zeta_q(2)}\right) = 3\frac{q+1}{q}\frac{q^{g+3}}{\zeta_q(2)|\F_g|} \frac{n\pi_q(n)}{q^{n/2}} + O\left(\frac{1}{g^2}\right).$$

Therefore if $n<2g-C\log_q(g)$, we can rewrite \eqref{errterm2} as
\begin{align}\label{errterm3}
\begin{split}
 \frac{3nq^{g+3-n/2}}{4|\F_g|}\left(\sum_{\deg(P)=n}C(g;P) -4(1-\eta_g)\frac{q+1}{q}\frac{\pi_q(n)}{\zeta_q(2)} \right)+ O\left(\frac{1}{g^2} + \frac{q^{n/2-(1/2-\epsilon)g}}{g^2}\right)
\end{split}
\end{align}
where $\pi_q(n)$ is the number of primes of degree $n$.

So it remains to determine
\begin{align}\label{C(g)}
C(g) := \sum_{\deg(P)=n}C(g;P).
\end{align}

\section{Summing the Primes}

Recently David, Koukoulopoulos, and Smith \cite{DKS} studied sums of Euler products over rational primes and their applications to statistics of elliptic curves. We will use similar techniques as them and apply it to suming Euler products in the fucntion field setting.

\subsection{Framework}\label{framework}

As in \cite{DKS}, we begin with an arbitrary framework. For every prime $Q$ suppose we have a function $\delta(u;Q)$ such that
$$\delta(u;Q) = (c_1\chi_1(Q) + \dots + c_n\chi_n(Q))u^{\deg(Q)} + O\left(|u|^{(1+\eta)\deg(Q)}\right)$$
where $c_i\in\mathbb{C}$, $\chi_i(Q) = \epsilon_i^{\deg(Q)}\tilde{\chi}_i(Q)$ for some root of unity $\epsilon_i$, $\tilde{\chi}_i$ is a quadratic Dirichelt character modulo $D_i$ and $\eta$ is some positive constant.

Define the euler product
$$\Q_{\delta}(u) := \prod_Q \left( 1 + \delta(u;Q) \right).$$
Here the product is an infinite product over all the monic prime polymials in $\Ff_q[X]$. Then, since $\delta(u;Q)\ll u^{\deg(Q)}$ we see that $\Q_{\delta}(u)$ converges for $|u|<q^{-1}$.

\begin{lem}\label{framelem1}

$\Q_{\delta}(u)$ can be analytically extended to the region $|u|<\min(q^{-1/(1+\eta)},q^{-1/2})$.

\end{lem}

\begin{proof}
By definition,
\begin{align*}
\Q_{\delta}(u) & = \prod_Q \left(1+(c_1\chi_1(Q) + \dots + c_n\chi_n(Q))u^{\deg(Q)} + O\left(u^{(1+\eta)\deg(Q)}\right)\right)\\
& = \prod_{i=1}^n \prod_Q\left(1 - \chi_i(Q) u^{\deg(Q)}\right)^{-c_i} \prod_Q \left(1 + O\left(u^{\min(1+\eta,2)\deg(Q)}\right)\right) \\
& = \prod_{i=1}^n L(\epsilon_iu,\tilde{\chi}_i)^{c_i} H(u).
\end{align*}
Since the $\tilde{\chi}_i$ are Dirichlet characters, we get that $L(\epsilon_iu,\tilde{\chi}_i)$ converges and is non-zero in the region $|u|<q^{-1/2}$. Moreover, by the Euler product expansion, we see that $H(u)$ converges for $|u|<\min(q^{-1/(1+\eta)},q^{-1/2})$ which proves the lemma.
\end{proof}

For simplicity we use the same notation, $\Q_{\delta}(u)$, to denote the analytic continuation as well.

For any $M$, define
$$\Q_{\delta}^{(M)}(u) = \prod_{\deg(Q)\leq M} \left(1+\delta(u;Q) \right).$$
Here, the product is over all monic prime polynomials in $\Ff_q[X]$ such that $\deg(Q)\leq M$. This is a finite product so it converges for all $u$.

We wish to approximate $\Q_{\delta}$ with $\Q_{\delta}^{(M)}$.

\begin{lem}\label{framelem2}
For $|u|<\min(q^{-1/(1+\eta)},q^{-1/2})$,
$$\Q_{\delta}(u) = \Q_{\delta}^{(M)}(u)\left\{1 + O\left(\frac{\left(q^{1/2}|u|\right)^M}{M} + \frac{\left(q|u|^{1+\eta}\right)^M}{M}\right)\right\}$$
where the implied constant depends on $c_i$ and $\chi_i$.

\end{lem}

\begin{proof}
We clearly have
$$\Q_{\delta}(u) = \Q_{\delta}^{(M)}(u) \prod_{\deg(Q)>M} \left(1 + \delta(u;Q)\right).$$

So it is enough to bound the tail of $\Q_{\delta}(u)$. We do this by examining its logarithm:
\begin{align}\label{alln}
\log \prod_{\deg(Q)>M} \left(1 + \delta(u;Q)\right) = \sum_{\deg(Q)>M} \sum_{n=1}^{\infty} \frac{(-1)^{n+1}\delta(u;Q)^n}{n}.
\end{align}
When $n=1$, we get
\begin{align}\label{n=1}
\sum_{\deg{Q}>M}\delta(u;Q) = \sum_{i=1}^n c_i \sum_{\deg(Q)>M}\chi_i(Q)u^{\deg(Q)} + O\left(\sum_{\deg{Q}>M}|u|^{(1+\eta)\deg(Q)}\right)
\end{align}
Applying \eqref{Riemhypbound} we get that
$$\sum_{\deg(Q)=m} \chi_i(Q) = O\left( \frac{\deg(D_i)}{m}q^{m/2} \right).$$
Therefore as long as $|u|< q^{-1/2}$, the main term in \eqref{n=1} is
\begin{align*}
\sum_{i=1}^n c_i \sum_{m=M+1}^{\infty}\sum_{\deg(Q)=m}\chi_i(Q)u^{\deg(Q)} & \ll \sum_{i=1}^n |c_i| \sum_{m=M+1}^{\infty}  \frac{\deg(D_i)}{m}q^{m/2} |u|^m\\
& \ll \sum_{i=1}^n |c_i|\deg(D_i) \sum_{m=M+1}^{\infty}  \frac{\left(q^{1/2}|u|\right)^m}{m} \\
& \ll \sum_{i=1}^n |c_i|\deg(D_i)  \frac{\left(q^{1/2}|u|\right)^M}{M}.
\end{align*}
Whereas as long as $|u|<q^{-1/(1+\eta)}$, applying \eqref{PPNT} the error term in $\eqref{n=1}$ is
$$\sum_{\deg{Q}>M}|u|^{(1+\eta)\deg(Q)} \ll \sum_{m=M+1}^{\infty} \frac{\left(q|u|^{1+\eta}\right)^m}{m} \ll \frac{\left(q|u|^{1+\eta}\right)^M}{M}$$

Using the bound $\delta(u;Q) \ll |u|^{\deg(Q)}$, when $n\geq 2$ in \eqref{alln} we have
\begin{align*}
\sum_{\deg(Q)>M} \sum_{n=2}^{\infty} \frac{(-1)^{n+1}\delta(u;Q)^n}{n} & \ll \sum_{n=2}^{\infty} \frac{1}{n}\sum_{m=M+1}^{\infty}\frac{\left(q|u|^n\right)^m}{m} \\
& \ll \frac{q^M}{M}\sum_{n=2}^{\infty} \frac{|u^{nM}|}{n} \ll \frac{(qu^2)^M}{M}
\end{align*}
as long as $|u|<q^{-1/2}$.

Therefore, we see that
\begin{align*}
\prod_{\deg(Q)>M} \left(1 + \delta(u;Q)\right) = e^{O\left(\frac{\left(q^{1/2}|u|\right)^M}{M}+\frac{\left(q|u|^{1+\eta}\right)^M}{M}\right)} = 1 + O\left(\frac{\left(q^{1/2}|u|\right)^M}{M}+\frac{\left(q|u|^{1+\eta}\right)^M}{M}\right)
\end{align*}
\end{proof}

\begin{cor}\label{framecor1}
For $|u|<\min(q^{-1/(1+\eta)},q^{-1/2})$,
$$\Q_{\delta}(u) = \Q_{\delta}^{(M)}(u)+ O\left(\frac{\left(q^{1/2}|u|\right)^M}{M} + \frac{\left(q|u|^{1+\eta}\right)^M}{M}\right)$$
where the implied constant depends on $c_i$ and $\chi_i$.
\end{cor}

\begin{proof}

The same method as in Lemma \ref{framelem2} shows that
$$\Q_{\delta}^{(M)}(u) = e^{O\left(\left(q^{1/2}|u|\right)^M+\left(q|u|^{1+\eta}\right)^M\right)}.$$
Therefore, as long as $|u|<\min(q^{-1/(1+\eta)},q^{-1/2})$ we have $\Q_{\delta}^{(M)}(u) = O(1)$ and the result follows from Lemma \ref{framelem2}.

\end{proof}

\subsection{Application of Framework}

Equations \eqref{Hpmdef} and \eqref{H0def} imply that
$$L(u,\chi_P^{\pm})^2H_{\pm}(u) = \prod_Q \left( 1 + 2\chi_P^{\pm}(Q)u^{\deg(Q)} - \left(1+2\chi_P^{\pm}(Q)\right)u^{2\deg(Q)} \right)$$
and
$$L(u,\chi_P^+)L(u,\chi_P^-)H_0(u) = \prod_Q \left( 1 + \left(\chi_P^{+}(Q)+\chi_P^{-}(Q)\right)u^{\deg(Q)} - \left(1+\chi_P^{+}(Q)+\chi_P^-(Q)\right)u^{2\deg(Q)} \right).$$

Therefore, if we define
$$\delta_{P,\pm}(u;Q) = 2\chi_P^{\pm}(Q)u^{\deg(Q)} - \left(1+2\chi_P^{\pm}(Q)\right)u^{2\deg(Q)} $$
and
$$\delta_{P,0}(u;Q) = \left(\chi_P^{+}(Q)+\chi_P^{-}(Q)\right)u^{\deg(Q)} - \left(1+\chi_P^{+}(Q)+\chi_P^-(Q)\right)u^{2\deg(Q)}$$
we get that
$$L(u,\chi_P^{\pm})^2H_{\pm}(u) = \Q_{\delta_{P,\pm}}(u)$$
and
$$L(u,\chi_P^+)L(u,\chi_P^-)H_0(u) = \Q_{\delta_{P,0}}(u).$$

\begin{prop}\label{frameapprop}
Let $* = +,-$ or $0$. Then for any $M>0$,
$$\sum_{\deg(P)=n} \Q_{\delta_{P,*}}(q^{-1}) = \frac{\pi_q(n)}{\zeta_q(2)} + O\left(\frac{q^{n-2M}}{n} + \frac{q^{n/2}M^3}{n} + \frac{q^n}{nMq^{M/2}}\right)$$
\end{prop}

\begin{proof}

Applying Corollary \ref{framecor1} we obtain
$$\Q_{\delta_{P,*}}(q^{-1}) = \Q_{\delta_{P,*}}^{(M)}(q^{-1}) + O\left(\frac{1}{Mq^{M/2}}\right).$$
Therefore by \eqref{PPNT},
$$\sum_{\deg(P)=n} \Q_{\delta_{P,*}}(q^{-1})= \sum_{\deg(P)=n}\Q_{\delta_{P,*}}^{(M)}(q^{-1}) + O\left(\frac{q^n}{nMq^{M/2}}\right)$$
so we will work only with the truncated product.

Let $*_1,*_2$ be either of $+$ or $-$ then we can write
\begin{align}\label{frameappeq1}
\sum_{\deg(P)=n}\Q_{\delta_{P,*}}^{(M)}(q^{-1}) = \sum_{\deg(P)=n}\prod_{\deg{Q}\leq M} \left( 1 + \frac{\chi_P^{*_1}(Q)+\chi_P^{*_2}(Q)}{|Q|} - \frac{1+\chi_P^{*_1}(Q)+\chi_P^{*_2}(Q)}{|Q|^2}\right)
\end{align}
where for any polynomial $F\in\Ff_q[X]$, $|F|=q^{\deg(F)}$.

Since the product is finite, we can expand it without worrying about conditional convergence. Denote $P^+(F)$ as the largest prime divisor of $F$. Then we can rewrite \eqref{frameappeq1} as
$$\sum_{\deg(P)=n} \sum_{\substack{\deg(P^+(F))\leq M \\ \deg(P^+(G))\leq M}} \mu^2(FG) \prod_{Q|F} \left( \frac{\chi_P^{*_1}(Q)+\chi_P^{*_2}(Q)}{|Q|}\right) \prod_{Q|G} \left(- \frac{1+\chi_P^{*_1}(Q)+\chi_P^{*_2}(Q)}{|Q|^2}\right)$$
$$ = \sum_{\deg(P)=n} \sum_{\substack{\deg(P^+(F))\leq M \\ \deg(P^+(G))\leq M}} \mu^2(FG) \frac{1}{|F|}\sum_{D|F} \chi_P^{*_1}(D)\chi_P^{*_2}(F/D) \frac{\mu(G)}{|G|^2}\sum_{D_1D_2|G}\chi_P^{*_1}(D_1)\chi_P^{*_2}(D_2)$$
$$ = \sum_{\substack{\deg(P^+(F))\leq M \\ \deg(P^+(G))\leq M}}  \frac{\mu^2(FG)\mu(G)}{|FG^2|}\sum_{D|F}\sum_{D_1D_2|G} \sum_{\deg(P)=n}\chi_P^{*_1}(DD_1)\chi_P^{*_2}(D_2F/D)  $$

Let $S_1$ be the partial sum where $FD_1D_2=1$ and $S_2$ the remainder. Then,
$$S_1 = \sum_{\deg(P^+(G))\leq M} \frac{\mu(G)}{|G|^2}\sum_{\deg(P)=n} 1 = \pi_q(n) \prod_{\deg{Q}\leq M} \left(1-\frac{1}{|Q|^2}\right).$$

By the same method as in Section \ref{framework} we see that
$$\prod_{\deg{Q}\leq M} \left(1-\frac{1}{|Q|^2}\right) = \prod_{Q} \left(1-\frac{1}{|Q|^2}\right) + O\left(\frac{1}{q^{2M}}\right) = \frac{1}{\zeta_q(2)}+ O\left(\frac{1}{q^{2M}}\right).$$
Therefore,
$$S_1 = \frac{\pi_q(n)}{\zeta_q(2)}+ O\left(\frac{q^{n-2M}}{n}\right).$$

We notice that for a fixed $F,D_1,D_2$,
$$\chi_P^{*_1}(DD_1)\chi_P^{*_2}(D_2F/D) = \pm \chi_P(FD_1D_2).$$
Further, by quadratic reciprocity we have
\begin{align*}
\chi_P(FD_1D_2) = \left(\frac{FD_1D_2}{P}\right) & = (-1)^{\frac{q-1}{2}\deg(FD_1D_2P)} \left(\frac{P}{FD_1D_2}\right) \\
&  = (-1)^{\frac{q-1}{2}\deg(FD_1D_2P)} \chi_{FD_1D_2}(P).
\end{align*}
The conditions imposed on $F,D_1$ and $D_2$ in the sum imply that $FD_1D_2$ will always be square-free. Thus, $\chi_{FD_1D_2}$ will be a non-trivial quadratic Dirichlet character modulo $FD_1D_2$ as long as $FD_1D_2\not=1$. Therefore for a fixed $F,D_1,D_2$, we apply \eqref{Riemhypbound} to get
$$\sum_{\deg(P)=n}\chi_P^{*_1}(DD_1)\chi_P^{*_2}(D_2F/D) = \pm \sum_{\deg(P)=n}\chi_{FD_1D_2}(P) \ll \frac{\deg(FD_1D_2)q^{n/2}}{n}.$$

Applying this bound we get
$$S_2 =  \sum_{\substack{\deg(P^+(F))\leq M \\ \deg(P^+(G))\leq M}}  \frac{\mu^2(FG)\mu(G)}{|FG^2|}\sum_{D|F}\sum_{\substack{D_1D_2|G \\ FD_1D_2\not=1}} \sum_{\deg(P)=n}\chi_P^{*_1}(DD_1)\chi_P^{*_2}(D_2F/D)$$
$$ \ll  \sum_{\substack{\deg(P^+(F))\leq M \\ \deg(P^+(G))\leq M}}  \frac{\mu^2(FG)}{|FG^2|}\sum_{D|F}\sum_{\substack{D_1D_2|G\\ FD_1D_2\not=1}}\frac{\deg(FD_1D_2)q^{n/2}}{n} $$
$$ \ll \frac{q^{n/2}}{n} \sum_{\deg(P+(F))\leq M} \frac{\mu^2(F)\deg(F)}{|F|}\left(\sum_{D|F}1\right)  \sum_{\deg(P+(G))\leq M} \frac{\mu^2(G)}{|G|^2} \sum_{D_1D_2|G}\deg(D_1D_2)$$

We can bound $\deg(D_1D_2)\ll |D_1D_2|^{\epsilon}$ for any $\epsilon>0$. Moreover, since $G$ is square-free we have $\sum_{D_1D_2|G}1 = 3^{\omega(G)}$ where $\omega(G)$ is the number of prime divisors of $G$. Hence,
$$\sum_{\deg(P+(G))\leq M} \frac{\mu^2(G)}{|G|^2} \sum_{D_1D_2|G}\deg(D_1D_2) \ll \sum_{\deg(P+(G))\leq M} \frac{\mu^2(G)3^{\omega(G)}}{|G|^{2-\epsilon}}$$
$$ \ll \prod_{\deg(Q)\leq M} \left(1+\frac{3}{|Q|^{2-\epsilon}}\right) \leq \prod_{Q} \left(1+\frac{3}{|Q|^{2-\epsilon}}\right) = O(1). $$

Again, since $F$ is square-free, we can write $\sum_{D|F}1 = 2^{\omega(F)}$ where $\omega(F)$ is the number of prime divisors of $F$ and we can use the more precise bound $\deg(F)\leq M|F|^{1/M}$. Hence,
$$\sum_{\deg(P+(F))\leq M} \frac{\mu^2(F)\deg(F)}{|F|}\left(\sum_{D|F}1\right) \leq M \sum_{\deg(P+(F))\leq M} \frac{\mu^2(F)2^{\omega(F)}}{|F|^{1-1/M}}$$
$$ = M \prod_{\deg(Q)\leq M} \left(1 + \frac{2}{|Q|^{1-1/M}}\right) \ll M \exp\left(\sum_{\deg(Q)\leq M} \frac{2}{|Q|^{1-1/M}} \right)$$
where the last inequality can be shown using the methods of Section \ref{framework}.

Finally, since $\deg(Q)\leq M$, we have $|Q|^{1/M} = 1+O\left(\deg(Q)/M\right)$ and
\begin{align*}
\sum_{\deg(Q)\leq M} \frac{1}{|Q|^{1-1/M}}&  = \sum_{\deg{Q}\leq M} \frac{1}{|Q|} + O\left(\frac{1}{M}\sum_{\deg{Q}\leq M} \frac{\deg(Q)}{|Q|}\right)\\
& = \sum_{m=1}^M \frac{\pi_q(m)}{q^m} + O\left(\frac{1}{M} \sum_{m=1}^M \frac{m\pi_q(m)}{q^m}\right)\\
& = \sum_{m=1}^M \frac{1}{m} + O\left(\sum_{m=1}^M \frac{1}{q^{m/2}} + \frac{1}{M}\sum_{m=1}^M 1\right)\\
& = \log(M) +O(1).
\end{align*}

Therefore, we conclude that $S_2 \ll \frac{q^{n/2}}{n} M^3$.

\end{proof}

\section{Proof of Theorem \ref{mainthm}}

The case of odd $n$ was already proved in Section \ref{Oddn}.

For even $n$, combining equations \eqref{traceform2}, \eqref{roots}, \eqref{nongenval3} and \eqref{errterm2} implies that for $n\geq3\log_q(g)$,
$$\langle\Tr(\Theta_C^n)\rangle_{\Hh_{g,\biq}} = -3 + \frac{3nq^{-n/2}}{|\F_g|} \sum_{\deg(P)=n} \sum_{(f_1,f_2,f_3)\in\F_g} \chi_P(f_1f_2) + o\left(1\right)$$
as $g\to\infty$. Equation \eqref{errterm3} states that if $n<2g-C\log_q(g)$, then the above double sum is equal to
\begin{align*}
\frac{3nq^{g+3-n/2}}{4|\F_g|}\left(\sum_{\deg(P)=n}C(P,g) - 4 (1-\eta_g)\frac{q+1}{q}\frac{\pi_q(n)}{\zeta_q(2)}\right) + O\left(\frac{1}{g^2} + \frac{q^{n/2-(1/2-\epsilon)g}}{g^2}\right)
\end{align*}
where $C(P;g)$ is defined in \eqref{C(P)}. Applying Proposition \ref{frameapprop} we get
\begin{align*}
\sum_{\deg(P)=n}C(P;g) & =  \frac{q+3}{q} \sum_{\deg(P)=n} L(q^{-1},\chi_P^+)^2H_{P,+}(q^{-1}) + \frac{q-1}{q} \sum_{\deg(P)=n} L(q^{-1},\chi_P^-)^2H_{P,-}(q^{-1}) \\
&  - 2(-1)^g\frac{q+1}{q} \sum_{\deg(P)=n} L(q^{-1},\chi_P^+)L(q^{-1},\chi_P^-)H_{P,0}(q^{-1}) \\
& = \frac{2q+2-(-1)^g(2q+2)}{q} \frac{\pi_q(n)}{\zeta_q(2)} + O\left(\frac{q^{n-2M}}{n} + \frac{q^{n/2}M^3}{n} + \frac{q^n}{nMq^{M/2}}\right)\\
& = 4(1-\eta_g)\frac{q+1}{q}\frac{\pi_q(n)}{\zeta_q(2)}+ O\left(\frac{q^{n-2M}}{n} + \frac{q^{n/2}M^3}{n} + \frac{q^n}{nMq^{M/2}}\right).
\end{align*}
Hence, for any $M>0$,
$$\langle\Tr(\Theta_C^n)\rangle_{\Hh_{g,\biq}} = -3 +o\left(1\right) + O\left(\frac{q^{n/2-(1/2-\epsilon)g}}{g^2} + \frac{q^{\frac{n}{2}-2M}}{g^2} + \frac{M^3}{g^2} + \frac{q^{\frac{n-M}{2}}}{g^2}\right).$$
We wish to choose $M$ so that this error term tends to $0$ as $g$ tends to infinity for as wide a range of $n$ as possible. The term
$$\frac{q^{n/2-(1/2-\epsilon)g}}{g^2}$$
means that we must have $n<(1-\epsilon)g$ for some $\epsilon>0$.  The terms
$$ \frac{q^{\frac{n}{2}-2M}}{g^2} \mbox{ and } \frac{q^{\frac{n-M}{2}}}{g^2}$$
will tend to $0$ for large $n$ only if $M\geq n$. So setting $M=n$, we see that
$$\frac{M^3}{g^2} = \frac{n^3}{g^2}$$
will tend to $0$ only if $n=o(g^{2/3})$.

Therefore, if $n$ is even such that $3\log_q(g) \leq n = o(g^{2/3})$ we find that
$$\langle\Tr(\Theta_C^n)\rangle_{\Hh_{g,\biq}} = -3 + o(1)$$
as $g\to\infty$.

Finally, we recall that for $n$ even and $n\leq 2g$, we have
$$\int_{USp(2g)^3}\Tr(U^n)dU = -3$$
which finishes the proof.

\section{One Level Density}\label{density}

Recall, for any even test function $f$ we define
$$F(\theta) := \sum_{n\in\Z} f\left(2g\left( \frac{\theta}{2\pi}-n \right)\right).$$
Let $g(x) = f\left(2g\left(\frac{\theta}{2\pi}-x\right)\right)$, then by Poisson summation formula we have
$$F(\theta) = \sum_{n\in \Z} g(n) = \sum_{n\in\Z} \hat{g}(n) = \sum_{n\in\Z}\int_{-\infty}^{\infty} g(x) e^{-2\pi ixn}dx$$
$$ = \sum_{n\in\Z}\int_{-\infty}^{\infty} f\left(2g\left(\frac{\theta}{2\pi}-x\right)\right) e^{-2\pi ixn}dx = \frac{1}{2g}\sum_{n\in\Z}\int_{-\infty}^{\infty} f(y) e^{-2\pi i\left(\frac{-y}{2g} + \frac{\theta}{2\pi}\right)n}dy$$
$$ = \frac{1}{2g}\sum_{n\in\Z} e^{in\theta}\int_{-\infty}^{\infty} f(y) e^{-2\pi iy \frac{-n}{2g}}dy =\frac{1}{2g}\sum_{n\in\Z} e^{in\theta} \hat{f}\left(\frac{-n}{2g}\right) $$
$$ =\frac{1}{2g}\sum_{n\in\Z} e^{in\theta} \hat{f}\left(\frac{n}{2g}\right) $$
where the last equality comes from the fact that $f$ is even.

For any $U\in USp(2g)$ with eigenvalues $e^{i\theta_j}$, $j=1,\dots,2g$, we let
$$Z_F(U) = \sum_{j=1}^{2g} F(\theta_j).$$

Therefore, if $e^{i\theta_j}$ are the eigenvalues of $U$, we get
$$Z_F(U) =  \sum_{j=1}^{2g} F(\theta_j) = \sum_{j=1}^{2g} \frac{1}{2g}\sum_{n\in\Z} e^{in\theta_j} \hat{f}\left(\frac{n}{2g}\right)  $$
$$ = \hat{f}(0) + \frac{1}{g}\sum_{n=1}^{\infty} \hat{f}\left(\frac{n}{2g}\right) \Tr(U^n).$$

Therefore, if Theorem \ref{mainthm} holds for $n<2\alpha g$ and $f$ is chosen such that $\supp(\hat{f}) \subset (-\alpha,\alpha)$, then
\begin{align*}
\langle Z_F(\Theta_C)\rangle_{\Hh_{g,\biq}} & = \hat{f}(0) + \frac{1}{g}\sum_{n=1}^{2\alpha g} \hat{f}\left(\frac{n}{2g}\right) \langle\Tr(\Theta_C^n)\rangle_{\Hh_{g,\biq}} \\
& =  \hat{f}(0) + \frac{1}{g}\sum_{n=1}^{2\alpha g} \hat{f}\left(\frac{n}{2g}\right) \int_{USp(2g)^3}\Tr(U^n)dU + o(1)\\
& =  \hat{f}(0) + \frac{1}{g}\sum_{n=1}^{\infty} \hat{f}\left(\frac{n}{2g}\right) \int_{USp(2g)^3}\Tr(U^n)dU + o(1)\\
& = \int_{USp(2g)^3} Z_f(U) dU + o(1).
\end{align*}

\bibliography{HPFirstDraft}
\bibliographystyle{amsplain}

\end{document}